\documentclass[a4paper,11pt]{article}
\usepackage{geometry}
\usepackage[english]{babel}
\usepackage{amssymb}
\usepackage{enumitem}
\usepackage{latexsym}
\usepackage[all,cmtip]{xy}
\usepackage{amsmath,amssymb,amsthm}
\usepackage{comment}
\usepackage{tikz}
\usepackage{rotating}
\usepackage{marvosym}
\usepackage{adjustbox}
\usepackage{tabu}
\usepackage{etoolbox}

\title{Morita equivalence classes of principal blocks with elementary abelian defect groups of order $64$ \thanks{This paper is part of the work done by the author during his PhD at the University of Manchester, supported by a Manchester Research Scholar Award and a President’s Doctoral Scholar Award.}}
\author{
   Cesare Giulio Ardito\thanks{Department of Mathematics, City University of London, Northampton Square, London, EC1V 0HB, United Kingdom. Email: cesareg.ardito@gmail.com}
}

\newtheorem{theorem}{Theorem}[section]
\newtheorem*{theorem*}{Theorem}
\newtheorem{example}[theorem]{Example}
\newtheorem{method}[theorem]{Method}
\newtheorem{lemma}[theorem]{Lemma}
\newtheorem{corollary}[theorem]{Corollary}
\newtheorem{proposition}[theorem]{Proposition}
\newtheorem{remark}[theorem]{Remark}

\newtheorem{conjecture}[theorem]{Conjecture}

\newcommand{\Aut}{\operatorname{Aut}}
\newcommand{\Out}{\operatorname{Out}}

\newcommand{\Pic}{\mathop{\rm Pic}\nolimits}

\newcommand{\SL}{\operatorname{SL}}

\newcommand{\GL}{\operatorname{GL}}

\newcommand{\PSL}{\rm PSL}

\newcommand{\cO} {\mathcal{O}}
\newcommand{\cT} {\mathcal{T}}

\def\bigcp{\mathop{\mathchoice 
 {\hbox{\sf\Large\lower 0.1\baselineskip\hbox{Y}}}%
 {\hbox{\sf\large\lower 0.1\baselineskip\hbox{Y}}}%
 {\hbox{\sf\normalsize\lower 0.1\baselineskip\hbox{Y}}}%
 {\hbox{\sf\tiny\lower 0.1\baselineskip\hbox{Y}}}%
}}
\def\bigtimes{\mathop{\mathchoice 
 {\hbox{\sf\Large\lower 0.1\baselineskip\hbox{X}}}%
 {\hbox{\sf\large\lower 0.1\baselineskip\hbox{X}}}%
 {\hbox{\sf\normalsize\lower 0.1\baselineskip\hbox{X}}}%
 {\hbox{\sf\tiny\lower 0.1\baselineskip\hbox{X}}}%
}}

\def\Sym(#1){\mathop{\rm Sym}(#1)}
\def\Sym(#1){S_{#1}}
\def\diag(#1){\mathop{\rm diag}(#1)}

\newcounter{bl}

\begin{document}
\maketitle \vspace*{-1em}
\begin{abstract}
We classify the Morita equivalence classes of principal blocks with elementary abelian defect groups of order $64$ with respect to a complete discrete valuation ring with an algebraically closed residue field of characteristic two. 


\end{abstract}
\section{Introduction}
In the following let $\cO$ be a complete discrete valuation ring, with field of fractions $K$ of characteristic $0$ and residue field $k$, an algebraically closed field of characteristic two. We assume $K$ to be large enough for all finite groups considered in this paper. The triple $(K,\cO, k)$ is usually called \emph{$p$-modular system}.
We say that two algebras $A$ and $B$ are Morita equivalent if there is an equivalence of categories of $A$-modules and $B$-modules. Donovan's conjecture states that for each fixed $p$-group $D$ there are only finitely many Morita equivalence classes of blocks of finite groups with defect group $D$. The conjecture is still open, but it has been proved for abelian $2$-groups in \cite{eali18c} over $k$ and later in \cite{ealiei18} over $\cO$. 

A classification of all Morita equivalence classes of blocks with a smaller elementary abelian defect $2$-group $(C_2)^n$ has been done by Alperin for $n=1$ \cite{alp151}, various authors for $n=2$ \cite{erd82} \cite{lin94} \cite{cra11}, Eaton for $n=3, 4$ \cite{ea14} \cite{ea17} and the author for $n=5$ \cite{ar19}. We classify principal blocks with defect group $(C_2)^6$. The reason we restrict our attention to principal blocks is that, at the moment, the general method to classify blocks described in \cite{ar20} does not allow us to deal with arbitrary blocks in this case, mainly because in many cases we are unable to compute Picard groups of various blocks involved.

In Section 2 we establish our notation and define several algebraic objects. In Section 3 we list many classic reductions in modular representation theory, and prove some technical lemmas. In Section 4 we examine blocks of normal subgroups with odd index, describing a method originally developed in \cite{ar20} and used in \cite{ar19}. In Section 5 we prove our main theorem, and in Section 6 we examine nonprincipal blocks with defect group $(C_2)^6$ and give a list that we conjecture to be complete.

\begin{theorem}
\label{maintheorem6}
Let $G$ be a finite group with an elementary abelian Sylow $2$-subgroup $D$ of order $64$, and let $B$ be the principal block of $\cO G$. Then $B$ is Morita equivalent to the principal block of precisely one of the following groups: \vskip0.5em
\begin{tabular}{rlll}
\stepcounter{bl} (\roman{bl})& 	$(C_2)^6$ &\qquad \qquad & (inertial quotient $\{1\}$) \\
\stepcounter{bl}  (\roman{bl}) & 	$A_4 \times (C_2)^4$  &&  (i.q. $(C_3)_1$) \\
\stepcounter{bl}  (\roman{bl}) & 	$A_5 \times (C_2)^4$  &&  (i.q. $(C_3)_1$) \\
\stepcounter{bl}  (\roman{bl}) & 	$((C_2)^4 \rtimes C_3) \times (C_2)^2$  &&  (i.q. $(C_3)_2$) \\
\stepcounter{bl}  (\roman{bl}) & 	$(C_2)^6 \rtimes C_3$  &&  (i.q. $(C_3)_3$) \\
\stepcounter{bl}  (\roman{bl}) & 	$((C_2)^4 \rtimes C_5) \times (C_2)^2$   &&  (i.q. $C_5$) \\
\stepcounter{bl}  (\roman{bl}) & 	$((C_2)^3 \rtimes C_7) \times (C_2)^3$  &&  (i.q. $(C_7)_1$) \\
\stepcounter{bl}  (\roman{bl}) & 	$SL_2(8) \times (C_2)^3$   &&  (i.q. $(C_7)_1$) \\
\stepcounter{bl}  (\roman{bl}) & 	$(C_2)^6 \rtimes_2 C_7$  &&  (i.q. $(C_7)_2$) \\
\stepcounter{bl}  (\roman{bl}) & 	$(C_2)^6 \rtimes_4 C_7$  &&  (i.q. $(C_7)_3$) \\
\stepcounter{bl}  (\roman{bl}) & 	$(C_2)^6 \rtimes C_9$  &&  (i.q. $C_9$) \\
\stepcounter{bl}  (\roman{bl}) & 	$A_4 \times A_4 \times (C_2)^2$   &&  (i.q. $(C_3 \times C_3)_1$) \\
\stepcounter{bl}  (\roman{bl}) & 	$A_4 \times A_5 \times (C_2)^2$   &&  (i.q. $(C_3 \times C_3)_1$) \\
\stepcounter{bl}  (\roman{bl}) & 	$A_5 \times A_5 \times (C_2)^2$   &&  (i.q. $(C_3 \times C_3)_1$) \\
\stepcounter{bl}  (\roman{bl}) & 	$((C_2)^4 \rtimes C_3) \times A_4$  &&  (i.q. $(C_3 \times C_3)_3$) \\
\stepcounter{bl}  (\roman{bl}) & 	$((C_2)^4 \rtimes C_3) \times A_5$  &&  (i.q. $(C_3 \times C_3)_3$) \\
\stepcounter{bl}  (\roman{bl}) & 	$(C_2)^6 \rtimes (C_3 \times C_3)$  &\qquad \qquad &  (i.q. $(C_3 \times C_3)_2$) \\
\stepcounter{bl}  (\roman{bl}) & 	$((C_2)^4 \rtimes C_{5}) \times A_4$  &&  (i.q. $(C_{15})_1$) \\
\stepcounter{bl}  (\roman{bl}) & 	$((C_2)^4 \rtimes C_{5}) \times A_5$  &&  (i.q. $(C_{15})_1$) \\
\stepcounter{bl}  (\roman{bl}) & 	$((C_2)^4 \rtimes C_{15}) \times (C_2)^2$   &&  (i.q. $(C_{15})_2$) \\
\stepcounter{bl}  (\roman{bl}) & 	$SL_2(16) \times (C_2)^2$   &&  (i.q. $(C_{15})_2$) \\
\stepcounter{bl}  (\roman{bl}) & 	$(C_2)^6 \rtimes C_{15}$  &&  (i.q. $(C_{15})_3$) \\
\stepcounter{bl}  (\roman{bl}) & 	$((C_2)^3 \rtimes C_7) \times A_4 \times C_2$   &&  (i.q. $(C_{21})_1$)\\
\stepcounter{bl}  (\roman{bl}) & 	$((C_2)^3 \rtimes C_7) \times A_5 \times C_2$   &&  (i.q. $(C_{21})_1$) \\
\stepcounter{bl}  (\roman{bl}) & 	$\SL_2(8) \times A_4 \times C_2$  &&  (i.q. $(C_{21})_1$) \\
\stepcounter{bl}  (\roman{bl}) & 	$\SL_2(8) \times A_5 \times C_2$  &&  (i.q. $(C_{21})_1$) \\
\stepcounter{bl}  (\roman{bl}) & 	$(C_2)^6 \rtimes C_{21}$  &&  (i.q. $(C_{21})_2$) \\
\stepcounter{bl}  (\roman{bl}) & 	$((C_2)^3 \rtimes (C_7 \rtimes C_3)) \times (C_2)^3$   &&  (i.q. $(C_7 \rtimes C_3)_1$) \\
\stepcounter{bl}  (\roman{bl}) & 	$J_1 \times (C_2)^3$    &&  (i.q. $(C_7 \rtimes C_3)_1$) \\
\stepcounter{bl}  (\roman{bl}) & 	$\Aut(\SL_2(8)) \times (C_2)^3$   &&  (i.q. $(C_7 \rtimes C_3)_1$) \\
\stepcounter{bl}  (\roman{bl}) & 	$(C_2)^6 \rtimes (C_7 \rtimes C_3)$   &&  (i.q. $(C_7 \rtimes C_3)_2$) \\
\stepcounter{bl}  (\roman{bl}) & 	$(\SL_2(8) \times (C_2)^2) \rtimes C_3 \times C_2$   &&  (i.q. $(C_7 \rtimes C_3)_2$) \\
\stepcounter{bl}  (\roman{bl}) & 	$(C_2)^6 \rtimes (C_7 \rtimes C_3)$   &&  (i.q. $(C_7 \rtimes C_3)_3$) \\
\stepcounter{bl}  (\roman{bl}) & 	$(C_2)^6 \rtimes (C_7 \rtimes C_3)$   &&  (i.q. $(C_7 \rtimes C_3)_4$) \\
\stepcounter{bl}  (\roman{bl}) & 	$A_4 \times A_4 \times A_4$ &&  (i.q. $C_3 \times C_3 \times C_3$) \\
\stepcounter{bl}  (\roman{bl}) & 	$A_4 \times A_4 \times A_5$ &&  (i.q. $C_3 \times C_3 \times C_3$) \\
\stepcounter{bl}  (\roman{bl}) & 	$A_4 \times A_5 \times A_5$ &&  (i.q. $C_3 \times C_3 \times C_3$) \\
\stepcounter{bl}  (\roman{bl}) & 	$A_5 \times A_5 \times A_5$ &&  (i.q. $C_3 \times C_3 \times C_3$) \\
\stepcounter{bl}  (\roman{bl}) & 	$(C_2)^6 \rtimes 3^{1+2}_+$ && (i.q. $3^{1+2}_+$) \\
\stepcounter{bl}  (\roman{bl}) & 	$(C_2)^6 \rtimes 3^{1+2}_-$ && (i.q. $3^{1+2}_-$) \\
\stepcounter{bl}  (\roman{bl}) & 	$((C_2)^5 \rtimes C_{31}) \times C_2$   &&  (i.q. $C_{31}$) \\
\stepcounter{bl}  (\roman{bl}) & 	$\SL_2(32) \times C_2$  &&  (i.q. $C_{31}$) \\
\stepcounter{bl}  (\roman{bl}) & 	$((C_2)^4 \rtimes C_{15}) \times A_4$ && (i.q. $C_{15} \times C_3$) \\
\stepcounter{bl}  (\roman{bl}) & 	$((C_2)^4 \rtimes C_{15}) \times A_5$ && (i.q. $C_{15} \times C_3$) \\
\stepcounter{bl}  (\roman{bl}) & 	$SL_2(16) \times A_4$    &&  (i.q. $C_{15} \times C_3$) \\
\stepcounter{bl}  (\roman{bl}) & 	$SL_2(16) \times A_5$    &&  (i.q. $C_{15} \times C_3$) \\
\stepcounter{bl}  (\roman{bl}) & 	$((C_2)^3 \rtimes C_7) \times ((C_2)^3 \rtimes C_7)$  && (i.q. $C_7 \times C_7$) \\
\stepcounter{bl}  (\roman{bl}) & 	$((C_2)^3 \rtimes C_7) \times \SL_2(8)$ && (i.q. $C_7 \times C_7$) \\
\stepcounter{bl}  (\roman{bl}) & 	$\SL_2(8) \times \SL_2(8)$ && (i.q. $C_7 \times C_7$) \\
\stepcounter{bl}  (\roman{bl}) & 	$(C_2)^6 \rtimes C_{63}$  && (i.q. $C_{63}$) \\
\stepcounter{bl}  (\roman{bl}) & 	$\SL_2(64)$  && (i.q. $C_{63}$) 
\end{tabular} \\
\begin{tabular}{rlll}
\stepcounter{bl}  (\roman{bl}) & 	$(C_2)^6 \rtimes (C_7 \rtimes C_9)$ && (i.q. $ C_7 \rtimes C_9$) \\
\stepcounter{bl}  (\roman{bl}) & $((C_2)^3) \rtimes (C_7 \rtimes C_3)) \times A_4 \times C_2$   &&  (i.q. $((C_7 \rtimes C_3) \times C_3)_1$) \\
\stepcounter{bl}  (\roman{bl}) & $((C_2)^3) \rtimes (C_7 \rtimes C_3)) \times A_5 \times C_2$   &&  (i.q. $((C_7 \rtimes C_3) \times C_3)_1$) \\
\stepcounter{bl}  (\roman{bl}) & $J_1 \times A_4 \times C_2$   &&  (i.q. $((C_7 \rtimes C_3) \times C_3)_1$) \\
\stepcounter{bl}  (\roman{bl}) & $J_1 \times A_5 \times C_2$   &&  (i.q. $((C_7 \rtimes C_3) \times C_3)_1$) \\
\stepcounter{bl}  (\roman{bl}) & $\Aut(\SL_2(8)) \times A_4 \times C_2$   &&  (i.q. $((C_7 \rtimes C_3) \times C_3)_1$) \\
\stepcounter{bl}  (\roman{bl}) & $\Aut(\SL_2(8)) \times A_5 \times C_2$   &&  (i.q. $((C_7 \rtimes C_3) \times C_3)_1$) \\
\stepcounter{bl}  (\roman{bl}) &  $((C_2)^6 \rtimes ((C_7 \rtimes C_3) \times C_3)$   &&  (i.q. $((C_7 \rtimes C_3) \times C_3)_2$) \\
\stepcounter{bl}  (\roman{bl}) & 	$A_4 \wr C_3$ && (i.q. $C_3 \wr C_3$) \\
\stepcounter{bl}  (\roman{bl}) & 	$A_5 \wr C_3$ && (i.q. $C_3 \wr C_3$) \\
\stepcounter{bl}  (\roman{bl}) & $((C_2)^3) \rtimes (C_7 \rtimes C_3)) \times ((C_2)^3) \rtimes C_7)$   &&  (i.q. $(C_7 \rtimes C_3) \times C_7$) \\
\stepcounter{bl}  (\roman{bl}) & $((C_2)^3) \rtimes (C_7 \rtimes C_3)) \times \SL_2(8)$   &&  (i.q. $(C_7 \rtimes C_3) \times C_7$) \\
\stepcounter{bl}  (\roman{bl}) & $J_1 \times ((C_2)^3) \rtimes C_7)$   &&  (i.q. $(C_7 \rtimes C_3) \times C_7$) \\
\stepcounter{bl}  (\roman{bl}) & $J_1 \times \SL_2(8)$   &&  (i.q. $(C_7 \rtimes C_3) \times C_7$) \\
\stepcounter{bl}  (\roman{bl}) & $\Aut(\SL_2(8)) \times ((C_2)^3) \rtimes C_7)$   &&  (i.q. $(C_7 \rtimes C_3) \times C_7$) \\
\stepcounter{bl}  (\roman{bl}) & $\Aut(\SL_2(8)) \times \SL_2(8)$   &&  (i.q. $(C_7 \rtimes C_3) \times C_7$) \\
\stepcounter{bl}  (\roman{bl}) & 	$(C_2)^6 \rtimes ((C_7 \rtimes C_7) \rtimes C_3)$ && (i.q. $((C_7 \rtimes C_7) \rtimes_4 C_3)$) \\ 
\stepcounter{bl}  (\roman{bl}) & 	$(\SL_2(8) \times ((C_2)^3 \rtimes C_7)) \rtimes C_3$ && (i.q. $((C_7 \rtimes C_7) \rtimes_4 C_3)$) \\ 
\stepcounter{bl}  (\roman{bl}) & 	$(\SL_2(8) \times \SL_2(8)) \rtimes C_3$ && (i.q. $((C_7 \rtimes C_7) \rtimes_4 C_3)$) \\ 
\stepcounter{bl}  (\roman{bl}) & 	$(C_2)^6 \rtimes ((C_7 \rtimes C_7) \rtimes C_3)$ && (i.q. $((C_7 \rtimes C_7) \rtimes_5 C_3)$) \\ 
\stepcounter{bl}  (\roman{bl}) & 	$((C_2)^5 \rtimes (C_{31} \rtimes C_5)) \times C_2$    &&  (i.q. $C_{31} \rtimes C_5$) \\
\stepcounter{bl}  (\roman{bl}) & 	$\Aut(\SL_2(32)) \times C_2$ &&  (i.q. $C_{31} \rtimes C_5$) \\
\stepcounter{bl}  (\roman{bl}) & 	$(C_2)^6 \rtimes (C_{63} \rtimes C_3)$ && (i.q. $C_{63} \rtimes C_3$) \\
\stepcounter{bl}  (\roman{bl}) & 	$\SL_2(64) \rtimes C_3$ && (i.q. $C_{63} \rtimes C_3$) \\
\stepcounter{bl}  (\roman{bl}) & $(C_2)^6 \rtimes ((C_7 \rtimes C_3) \times (C_7 \rtimes C_3))$ && (i.q. $(C_7 \rtimes C_3) \times (C_7 \rtimes C_3)$) \\
\stepcounter{bl}  (\roman{bl}) & $((C_2)^3) \rtimes (C_7 \rtimes C_3)) \times J_1$   &&  (i.q. $(C_7 \rtimes C_3) \times (C_7 \rtimes C_3)$) \\
\stepcounter{bl}  (\roman{bl}) & $((C_2)^3) \rtimes (C_7 \rtimes C_3)) \times \Aut(\SL_2(8))$   &&  (i.q. $(C_7 \rtimes C_3) \times (C_7 \rtimes C_3)$) \\
\stepcounter{bl}  (\roman{bl}) & 	$J_1 \times J_1$ && (i.q. $(C_7 \rtimes C_3) \times (C_7 \rtimes C_3)$) \\
\stepcounter{bl}  (\roman{bl}) & 	$J_1 \times \Aut(\SL_2(8))$ && (i.q. $(C_7 \rtimes C_3) \times (C_7 \rtimes C_3)$) \\
\stepcounter{bl}  (\roman{bl}) & 	$\Aut(\SL_2(8)) \times \Aut(\SL_2(8))$ && (i.q. $(C_7 \rtimes C_3) \times (C_7 \rtimes C_3)$) \\
\end{tabular} \vskip1em
\noindent Moreover, if a block $C$ of $\cO H$ for a finite group $H$ is Morita equivalent to $B$, then the defect group of $C$ is isomorphic to $D$. If $C$ is the principal block of $\cO H$, then the inertial quotient of $C$ is isomorphic to the inertial quotient of $B$, and they have the same action on $D$.
\end{theorem} 
Detailed data about each class of blocks is available on the Block Library \cite{ea20}.

\section{Notation and definitions} 
For a finite group $G$ and a commutative ring $R$ we denote the group algebra as $RG$. In the following, $R$ will either be $\cO$ or $k$, which are as specified above. For a block $B$ of $\cO G$ with defect group $D$, given a block $e$ of $C_G(D)$ such that the Brauer correspondent $e^G=B$ (often called a \emph{root} of $B$), we define $N_G(D, e)/C_G(D)$ to be its inertial quotient (usually denoted as $E$). It is always a $p'$-group, and if the block $B$ is nilpotent then $E=1$. We denote the idempotent that specifies the block $B$ as $e_B$; in other words, $B = \cO G e_B$. Moreover, we denote the number of irreducible characters of $KG$ contained in the block $B$ as $k(B)$, and the number of irreducible Brauer characters of $kG$ (equivalently, of simple $kG$-modules) as $l(B)$. The unique block of $\cO G$ that contains the trivial module is called the \emph{principal} block, and we denote it as $B_0(\cO G)$.

Two $R$-algebras (in particular, blocks) $B$ and $C$ are \emph{Morita equivalent} if their module categories are equivalent as $R$-linear categories. A more explicit, equivalent definition is the existence of a $(B,C)$-bimodule $M$ and of a $(C,B)$-bimodule $N$ such that $M \otimes_C N \cong B$ as $(B,B)$-bimodules and $N \otimes_B M \cong C$ as $(C,C)$-bimodules. Using bimodules to characterize Morita equivalences allows us to examine stronger equivalences by looking at special properties of said bimodules.

Two algebras are \emph{basic Morita equivalent} if there is a Morita equivalence realised by bimodules with endopermutation source. Two algebras are \emph{source algebra equivalent} (or \emph{Puig equivalent}) if the Morita equivalence is realised by bimodules with trivial source. 

If two blocks $B$ and $C$ are Morita equivalent over $\cO$, then $k(B)=k(C)$ and $l(B)=l(C)$. Moreover, the centers $Z(B)$ and $Z(C)$ are isomorphic, and $B$ and $C$ have the same Cartan matrices and decomposition matrices \cite{lin18}. The defect groups of $B$ and $C$ have the same order, exponent and $p$-rank, so in particular they are isomorphic when one of them is elementary abelian. 
If two blocks $B$ and $C$ are basic Morita equivalent then, in addition to the invariants listed above, the blocks always have isomorphic defect groups and the same fusion system (meaning, in particular, that they have the same inertial quotient). Finally, if two blocks are source algebra equivalent then they have isomorphic source algebras. 

Given a block of $\cO G$, there is a unique block of $kG$ that corresponds to it, via the map $\cO G \to kG$, $B \to \overline{B} := B \otimes_\cO k$. In particular, a Morita equivalence between blocks of $\cO G$ and $\cO H$ induces a Morita equivalence between the corresponding blocks of $kG$ and $kH$. The converse is not known to hold, so a classification over $\cO$ is stronger. There is, however, no known counterexample.

The Picard group of a block $B$ is defined as the group of $(B,B)$-bimodules that induce a self-Morita equivalence of $B$. The group operation is given by the tensor product. In general, blocks of $kG$ can have infinite Picard groups but, as shown in \cite{eisele19}, blocks of $\cO G$ always have a finite Picard group. By definition, the Picard group of a block is invariant under Morita equivalences. There is always an injective map from the outer automorphism group of a block to the Picard group of said block, which can be used to control the outer automorphism group of a block of which only the Morita equivalence class is known. We are also interested in a particular subgroup: we define $\cT(B) \leq \Pic(B)$ to be the subgroup of the bimodules with trivial source. The main theorem of \cite{bkl17} gives a bound on the size of this subgroup. Note that, by definition, $\cT(B)$ is invariant under source algebra equivalences.

\section{Reductions and technical lemmas}
Given a group $G$ and a normal subgroup $N$, we say that a block $B$ of $\cO G$ covers a block $b$ of $\cO N$ if there is a module in $B$ such that the decomposition of its restriction to $\cO N$ has a summand in $b$, or equivalently if $e_B e_b \neq 0$. This relation is the main tool we use to obtain our classification, as the structures of $B$ and $b$ are closely related: as shown in \cite[15.1]{alp151}, the blocks covered by $B$ are in a single $G$-orbit ($G$ acts by conjugation), and the defect groups of $b$ are of the form $D \cap N$ where $D$ is a defect group of $B$. We remark that, since the restriction of the trivial module to a normal subgroup is still a trivial module, the principal block of $G$ always covers the principal block of $N$.

The following are standard reductions in modular representation theory.
\begin{theorem}[Fong's First Reduction \cite{lin18}] \label{fong16}
Let $G$ be a finite group, and let $N$ be a normal subgroup of $G$. Let $b$ be a block of $\cO N$ and $B$ be a block of $\cO G$  that covers $b$. Let $H$ be the stabiliser of $b$ in $G$ acting by conjugation. Then there is a one-to-one correspondence between the blocks of $\cO G$ that cover $b$ and the blocks of $\cO H$ that cover $b$, where each block is source algebra equivalent to its correspondent.  In particular, if $B$ is principal then so is its correspondent.\end{theorem}
\begin{proof}
The first claim can be found, among many others, in Proposition 6.8.3 in \cite{lin18}. The fact that $C$ is a principal block follows from the uniqueness of $C$ and Brauer's third main theorem.
\end{proof}

We use the following version of Fong's Second Reduction when $G$ has a normal $p'$-subgroup.
\begin{theorem}[\cite{kupu90}] 
Let $G$ be a finite group and $N \lhd G$. Let $B$ be a block of $\cO G$ with defect group $D$ that covers a $G$-stable nilpotent block $b$ of $\cO N$ with defect group $D \cap N$. Then there are finite groups $M \lhd L$ such that $M \cong D \cap N$, $L/M \cong G/N$, there is a subgroup $D_L \leq L$ with $D_L \cong D$ and $M \leq D_L$, and there is a central extension $\tilde{L}$ of $L$ by a $p'$-group, and a block $\tilde{B}$ of $\cO \tilde{L}$ which is Morita equivalent to $B$ and has defect group $\tilde{D} \cong D_L \cong D$.
If $B$ is the principal block, then so is $\tilde{B}$.
\end{theorem}
\begin{proof}
Guidance on the extraction of this result from \cite{kupu90} can be found in \cite[2.2]{ea14}, and the last claim is proved in \cite[2.2]{ea17}.
\end{proof}
In particular, we use the following immediate consequence:
\begin{corollary}[\cite{ea17}] \label{kulshammerpuigcorollary6}
Let $G$ be a finite group, let $N \lhd G$ with $N \not\leq Z(G)O_p(G)$. Let $B$ be a quasiprimitive $p$-block of $\cO G$ covering a nilpotent block $b$ of $\cO N$. Then there is a finite group $H$ with $[H:O_{p'}(Z(H))] < [G:O_{p'}(Z(G))]$ and a block $B_H$ with isomorphic defect group to the one of $B$, such that $B_H$ is Morita equivalent to $B$.
\end{corollary}

Given a finite group $G$, a block $B$ of $\cO G$ with defect group $D$ is \emph{nilpotent covered} if these exists a group $\tilde{G}$ and a nilpotent block $\tilde{B}$ of $\cO \tilde{G}$ such that $G \lhd \tilde{G}$ and $\tilde{B}$ covers $B$. A block $B$ is said to be \emph{inertial} if it is basic Morita equivalent to its Brauer correspondent in $N_G(D)$. The following proposition relates these concepts.
\begin{lemma}[\cite{pu11}, \cite{zhou16}]  \label{nilpotentcovered6}
Let $G$ be a finite group and let $N \lhd G$. Let $b$ be a $p$-block of $\cO N$ covered by a block $B$ of $\cO G$. Then:
\begin{enumerate}[label=(\arabic*)] \setlength\itemsep{0em}
\item If $B$ is inertial, then $b$ is inertial.
\item If $b$ is nilpotent covered, then $b$ is inertial.
\item If $p$ does not divide $[G:N]$ and $b$ is inertial, then $B$ is inertial.
\item If $b$ is nilpotent covered, then it has abelian inertial quotient.
\end{enumerate}
\end{lemma}
\begin{proof}(1), (2) are Theorem 3.13, and (4) is Corollary 4.3 in \cite{pu11}. (3) is the main theorem of \cite{zhou16}.
\end{proof}

Given a finite group $G$ and a normal subgroup $N$, and given a block $b$ of $\cO N$, the group $G$ acts by conjugation on the block $\overline{b}$ of $kN$. We define the subgroup $G[b] \leq G$ as the kernel of the map $G \to \Out(\overline{b})$. By definition, $N \leq G[b]$, and $G[b]$ is a normal subgroup of $G$. 

\begin{lemma} \label{inneraut6}
Let $G$ be a finite group and $B$ a block of $\cO G$ with defect group $D$. Let $N$ be a normal subgroup of $G$ that contains $D$, and suppose that $B$ covers a $G$-stable block $b$ of $\cO N$. Let $\hat{b}$ be a block of $\cO G[b]$ covered by $B$. Then: 
\begin{enumerate}[label=(\arabic*)]\setlength\itemsep{0em}
\item $b$ is source algebra equivalent to $\hat{b}$.
\item $B$ is the unique block of $\cO G$ that covers $\hat{b}$.
\end{enumerate}
\end{lemma}
\begin{proof}
From \cite[2.2]{kekoli} there is a source algebra equivalence between $\overline{\hat{b}}$ and $\overline{b}$, that lifts to $\hat{b}$ and $b$ by \cite[7.8]{pu88b}.
Part (2) follows from 3.5 in \cite[3.5]{murai12}.
\end{proof}
We can define $G[b]_\cO$, the subgroup of elements that act as inner automorphisms on $b$. Note that, using the canonical map $\cO G \to kG$, $G[b]_\cO \leq G[b]$.

The following fact about principal blocks covering each other is useful.
\begin{lemma} \label{iqmagic6}
Let $G$ be a finite group, and let $N \lhd G$ such that $p \not|\; [G:N]$. Let $E_b$ be the inertial quotient of the principal block of $\cO N$, and $E_B$ the inertial quotient of the principal block of $\cO G$. Then $E_b$ is a normal subgroup of $E_B$.
\end{lemma}
\begin{proof}
Recall that for a principal block, the inertial quotient is just $E= N_G(D)/C_G(D)$. In fact, the root $e$ of $B_0(\cO G)$ in $C_G(D)$ contains the restriction of the trivial module, which is irreducible and $G$-stable.

Consider the map $N_N(D) \to N_G(D)/C_G(D)$ defined by $n \mapsto nC_G(D)$. Since $nC_G(D) = 1$ only if $n \in (N \cap C_G(D))$, the kernel of this homomorphism is $C_N(D)$. By the first isomorphism theorem, $E_b \leq E_B$. Now $N_N(D)/C_N(D) \cong N_N(D)C_G(D)/C_G(D)$ and, since $N_N(D) \lhd N_G(D)$ and $C_G(D) \lhd N_G(D)$, we are done.
\end{proof}

Now we focus on the case when $p=2$, and $D= (C_2)^6$. We need to examine blocks with a smaller elementary abelian defect group, and in particular we highlight this result on blocks with a Klein four defect group:
\begin{proposition}[\cite{erd82}, \cite{lin94}, \cite{cra11}]\label{classification6}
Let $G$ be a finite group and let $B$ be a block of $\cO G$ with defect group $D = (C_2)^2$. Then $B$ is source algebra equivalent to $\cO (C_2)^2$, $\cO A_4$ or $B_0(\cO A_5)$. 
\end{proposition}

Now we list the possible inertial quotients for a block with defect group $(C_2)^6$ by looking at conjugacy classes of subgroups of $\GL_6(2)$ of odd order, computed with Magma \cite{magma}.

For each listed group, the group algebra $\cO(D\rtimes E)$ is a block, and for each different action of $E$ this block lies in its own corresponding Morita equivalence class.

We use the notation $F_{21} = C_7 \rtimes C_3$ and we denote $\operatorname{SmallGroup}(147,j)$ as $(C_7)^2 \rtimes_j C_3$ to distinguish the two different nontrivial semidirect products, that both arise as possible inertial quotients. In order to distinguish different actions of subgroups in the same isomorphism class, that correspond to conjugacy classes of subgroups in $\GL_6(2)$, we denote them as follows: \begin{itemize} \setlength\itemsep{0em}
\item $(C_3)_1$ is the one with $|C_D(E)|=2^4$.
\item $(C_3)_2$ is the one with $|C_D(E)|=2^2$.
\item $(C_3)_3$ is the one that fixes no nontrivial element of $D$. It is generated by the $21^{\operatorname{st}}$ power of the Singer cycle in $\GL_6(2)$.
\item $(C_3)^2_1$ is the one with $|C_D(E)|=2^2$. 
\item $(C_3)^2_2$ is the one realised by the unique subgroup $C_3 \times C_3$ of $C_{63} \rtimes C_3$.
\item $(C_3)^3_2$ is the one realised in $((C_2)^4 \rtimes (C_3)_2) \times ((C_2)^2 \rtimes (C_3)_1)$. 
\item $(C_7)_1$ is the one with $|C_D(E)|=2^3$.
\item $(C_7)_2$ is generated by the $9^{\operatorname{th}}$ power of the Singer cycle in $\GL_6(2)$.
\item $(C_7)_3$ is a free action on $D$ defined as follows: let $x$ be a Singer cycle of $\GL_6(2)$, and let $y=x^9$. Let $D=D_1 \times D_2$ with $D_i = (C_2)^3$. Then $(C_7)_3$ is defined as $C_7$ acting as $x$ on $D_1$, and as $x^3$ on $D_2$.
\item $(C_{15})_1$ is the one with $C_D(E)=1$, but $C_D(C_3)=(C_2)^4$.
\item $(C_{15})_2$ is the one with $C_D(E) = (C_2)^2$.
\item $(C_{15})_3$ is the one with $C_D(E)=1$, and $C_D(C_3)=1$.
\item $(C_{21})_1$ is the one with $|C_D(E)|=2^3$.
\item $(C_{21})_2$ is the one that fixes no nontrivial element of $D$. It is generated by the $3^{\operatorname{rd}}$ power of the Singer cycle in $\GL_6(2)$.
\item $(F_{21})_1$ is the one with $|C_D(E)|=2^3$.
\item $(F_{21})_2$ is the one with $|C_D(E)|=2$. 
\item $(F_{21})_3$ is the one where the unique subgroup $C_7$ acts as $(C_7)_2$. 
\item $(F_{21})_4$ is the one where the unique subgroup $C_7$ acts as $(C_7)_3$. 
\item $(F_{21} \times C_3)_1$ is the one where the unique subgroup $C_{21}$ acts as $(C_{21})_1$.
\item $(F_{21} \times C_3)_2$ is the one where the unique subgroup $C_{21}$ acts as $(C_{21})_2$.
\end{itemize}
We represent the relations between inertial quotients as a diagram: there is an arrow with a dotted line between two nontrivial subgroups of odd order $H, K \leq \GL_6(2)$ whenever $K$ has a subgroup $H'$ that is isomorphic to $H$ and whose action on $D$ is the same as the one of $H$, and $[K:H']$ is an odd prime. The line is continuous if $H'$ is a normal subgroup of $K$.

\includegraphics[width=0.97\textwidth]{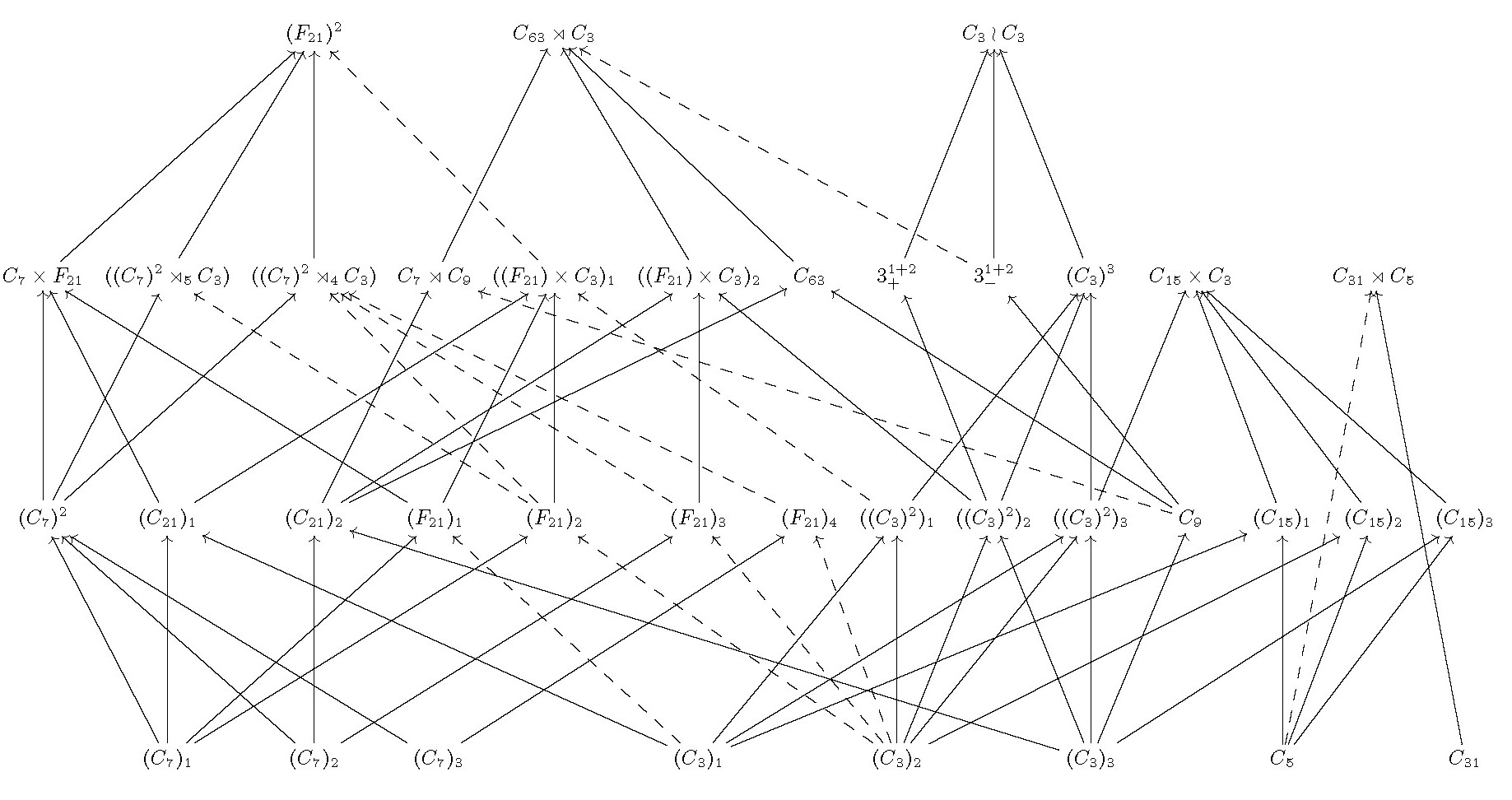}


Fong and Harris have shown in \cite{fongharris} that, when $p=2$, principal blocks with an abelian defect group are perfectly isometric to their Brauer correspondent. In particular, this enables us to determine various invariants that we use to simplify the classification process.
\begin{lemma} \label{sambalecalcs6}
Let $G$ be a finite group with an elementary abelian Sylow $2$-subgroup $D$, and let $B$ be the principal block of $\cO G$ (or $kG$). Let $E$ be the inertial quotient of $B$. Then the number of isomorphism classes of simple modules $l(B)$ is equal to the number of irreducible characters of $\cO E$, $k(E)$.

Explicitly, when $D=(C_2)^6$ all the possible pairs $(E,l(B))$ are: $(1,1)$, $(C_3,3)$, $(C_5,5)$, $(C_7,7)$, $(C_9,9)$, $(C_3 \times C_3,9)$, $(C_{15},15)$, $(C_7 \rtimes C_3,5)$, $(C_{21},21)$, $((C_3)^3,27)$, $(3^{1+2}_+,11)$, $(3^{1+2}_-,11)$, $(C_{31},31)$, $(C_{15} \times C_3,45)$, $(C_7 \times C_7,49)$, $((C_7 \rtimes C_3)\times C_3,15)$, $(C_7 \rtimes C_9,15)$, $(C_{63},63)$, $(C_3 \wr C_3,17)$, $((C_7 \times C_7) \rtimes_2 C_3),19)$, $((C_7 \times C_7) \rtimes_4 C_3),19)$, $((C_7 \rtimes C_3) \times C_7,35)$, $(C_{31} \rtimes C_5,11)$, $(C_{63}\rtimes C_3,29)$, $((C_7 \rtimes C_3)^2,25)$.
\end{lemma}
\begin{proof}
From the main theorem of \cite{fongharris}, $B$ is perfectly isometric to its Brauer correspondent $b$, the principal block of $\cO N_G(D)$. Note that, by definition, $B$ and $b$ have the same inertial quotient. From the main theorem of \cite{kul85} and \cite[6.14.1]{lin18}, $b$ is source algebra equivalent to a twisted group algebra $\cO_\alpha (D \rtimes E)$. From the discussion in \cite[2.5]{usami97}, since $b$ is the principal block then $\alpha=1$. In particular $l(B) = l(b) = l(\cO (D \rtimes E)) = k(E)$. The result follows by computing the number of irreducible ordinary characters for each possible isomorphism class of inertial quotients.
\end{proof}

Due to a theorem of Walter, we know precisely what nonabelian finite simple groups possess an abelian Sylow $2$-subgroup, and hence a principal block with a defect group we are interested into. Note that, perhaps unsurprisingly, every abelian defect $2$-group that occurs in finite simple groups is elementary abelian.

\begin{proposition}[\cite{walter69}] \label{walter}
Let $G$ be a nonabelian finite simple group with an abelian Sylow $2$-subgroup $P$. Then one of the following occurs:
\begin{enumerate}[label=(\roman*)]\itemsep0em
\item $G = \PSL_2(q)$ with $q \equiv 3, 5 \pmod 8$, $P = (C_2)^2$ and $N_G(P)\cong A_4$.
\item $G = \PSL_2(2^n)$ for $n \in \mathbb{N}$, $P= (C_2)^n$ and $N_G(P) \cong (C_2)^n \rtimes C_{2^n-1}$.
\item $G = {}^2G_2(3^{2n+1})$ for $n \in \mathbb{N}$, $P=(C_2)^3$ and $N_G(P) \cong C_2 \times A_4$.
\item $G= J_1$, $P=(C_2)^3$ and $N_G(P)=(C_2)^3 \rtimes (C_7 \rtimes C_3)$.
\end{enumerate}
\end{proposition}

Blocks of quasisimple groups with an abelian defect group when $p=2$ have been classified in \cite[6.1]{ekks}. However, our methods are not currently able to deal with nonprincipal blocks in full generality, which is why we use Walter's theorem instead.


\section{Crossed products and block chains}
We recall the key concepts from \cite{kul95}. Given a finite group $G$ and a ring with identity $A$, $A$ is a $G$-graded ring if there is a decomposition $A=\bigoplus_{g \in G} A_g$ as additive subgroups such that $A_g A_h \subseteq A_{gh}$, and $A_1=R$ is a subring of $A$ containing $1$.  A $G$-graded ring $A$ is called a crossed product of $A_1$ with $G$ if for any $g \in G$, $A_g$ contains at least one unit. We call two $G$-graded rings $A$ and $B$ weakly equivalent if there is an isomorphism of rings $\phi: A \to B$ such that $\phi(A_g) \subseteq B_g$ for all $g \in G$. Moreover, we say they are equivalent if $\phi$ restricts to the identity map on $A_1 \cong B_1$. One of the main results of \cite{kul95} is to give a parametrization of all possible crossed products between a ring $R$ and a group $G$.

\begin{theorem} \label{crossedprodeqclass6}
The equivalence classes of crossed products of a ring $R$ with a group $G$ are parametrized by pairs $(\omega, \zeta)$, where $\omega: G \to \Out(R)$ is a homomorphism whose corresponding $3$-cocycle in $H^3(G, \mathcal{U}(Z(R)))$ is zero, and $\zeta \in H^2 (G, \mathcal{U}(Z(R)))$ where the action of $G$ on $\mathcal{U}(Z(R))$ is induced by $\omega$. Moreover, weak equivalence classes of crossed products correspond to orbits of $\Aut(R)$ on the set of possible $(\omega, \zeta)$.
\end{theorem}

Given a finite group $G$ and normal subgroup $N$ with index a prime $\ell \neq p$, and given a block $B$ of $\cO G$ that covers a $G$-stable block $b$ of $\cO N$, then $B$ is a crossed product of $b$ with $G/N$. Therefore, we can use the theorem above to determine all possibilities for $B$ for each given $b$. In most cases we will encounter, however, we only know $b$ up to Morita equivalence, and we need a way to extend this information to the possibilities for $B$. Let $f$ be a basic idempotent of $b$, i.e. an idempotent such that $fbf$ is a basic algebra of $b$. Then, as shown in \cite[4.2]{ar19}, $fBf$ is a crossed product of $fbf$ with $X$ and $fBf$ is Morita equivalent to $B$. The method we use is precisely the one developed in \cite[4.5]{ar19}. We give a quick summary and establish notation for the benefit of the reader, and refer to \cite{ar19} for the technical details.

\begin{method} \label{clubsuit6}
\normalfont Let $G$ be a finite group and $B$ be the principal block of $\cO G$ with defect group $D \cong (C_2)^6$. Let $N \lhd G$ with $[G:N]$ odd (so $G/N$ is solvable) and let $B$ cover a $G$-stable block $b$ of $\cO N$. Moreover, suppose that $G/N$ is supersolvable, that $C_G(N) \leq N$, and that $N=\ker(G \to \Out(b))$ where the map is given by the action induced by conjugation of $G$ on $b$. Since $[G:N]$ is odd, $B$ and $b$ share a defect group. 

For a basic idempotent $f$ for $b$, by \cite[4.2]{ar19} we know that $B$ is Morita equivalent to $fBf$, which is a crossed product of $fbf$ with $G/N$. Then we have the following injective homomorphisms:
\begin{equation*}
G/N \xrightarrow{\quad \alpha\quad} \Out_\star(N) \xrightarrow{\quad\beta\quad} \Out(b) \xrightarrow{\quad\gamma\quad} \Pic(b)=\Pic(fbf)=\Out(fbf)
\end{equation*}
where $\Out_\star(N) = \{ \Phi \in \Out(N) : \forall \phi \in \Phi , \; \phi(b)=b \}$. 

For $\phi \in \Aut(b)$, we define the $b$-$b$-bimodule $\;_\phi b$ as: ${}_\phi b = b$ as sets, and $x.m.y=\phi(x)my$ for $x,m,y \in b$. Then for any $g \in G$ the induced action $\tau_g \in \Aut(b)$ given by conjugation has trivial source, which means that ${}_{\tau_g} b \in \cT(b)$, so the image of $G/N$ under $\gamma \beta \alpha$ is contained in $\cT(b)$. Note, however, that $\cT(b)$ is not invariant under Morita equivalences. 

Since $G/N$ is supersolvable, we consider a chain of normal subgroups with prime indices,
\begin{equation*} 
N = N_0 \lhd N_1 \lhd \dots \lhd N_t = G 
\end{equation*}
and a corresponding chain of blocks $b_i$ of $\cO N_i$ such that $b_i$ covers $b_{i-1}$, with $b_t:=B$ and $b_0:=b$. We can determine all possibilities for $b_1$ by considering all crossed product weak equivalence classes specified by possible maps $\omega_1$ given by the action of $G/N$ on $b$. In fact, we can assume that $\zeta =1$ (with the notation of Theorem \ref{crossedprodeqclass6}) because $N_1/N$ is cyclic (see \cite[4.]{ea17}). 

For each distinct possibility for $b_1$, we can then consider $N_2$ and $b_2$ as a crossed product of $b_1$ with $N_2/N_1$. Possibilities for $\omega_2$ are controlled by the image of $G/N_1$ in a quotient of $\Pic(b)$, but this is not enough: in fact, while we know the possibilities for the isomorphism class of $N_2/N_1$, we do not automatically know its embedding in $\Pic(b_1)$. Hence, it is necessary to also compute $\Pic(b_1)$ to proceed. This is not merely a technical issue, but a concrete obstruction, and the reason why a cocycle $\zeta$ is needed other than $\omega$ in Theorem \ref{crossedprodeqclass6} to determine the equivalence class of a crossed product. An exception is when $G/N$ is not cyclic but still has trivial second cohomology, such as when $G/N \leq C_7 \rtimes C_3$, in which case we can ``jump'' from $N$ to $G$ directly by considering all possibilities for $\omega$ without considering a block chain.

We iterate the process to determine all possible block chains and, at the end, all possible Morita equivalence classes for $B$. 

Since we are only looking at principal blocks in this paper, we can use Lemma \ref{sambalecalcs6} together with Lemma \ref{iqmagic6} to detect the inertial quotient of $b$ from knowledge of $l(b)$, and then to exclude some of the classes determined for $b_1, \dots, B$ based on Lemma \ref{iqmagic6}. Note that in general the inertial quotient is not known to be invariant under Morita equivalences, which is why we need to use Lemma \ref{sambalecalcs6}. Further, note that analysis of crossed products does not distinguish principal blocks, and some classes can only occur as nonprincipal blocks (any interested reader may compare Lemma 4.11 in \cite{mck19} about arbitrary blocks with Lemma \ref{iqmagic6}). 
\end{method}

Now we apply Method \ref{clubsuit6} to the cases we are interested in. We label Morita equivalence classes as in Theorem \ref{maintheorem6}.

\begin{proposition}\label{33}
Let $G$ be a finite group and $B$ be the principal block of $\cO G$ with defect group $D= (C_2)^6$. Suppose that $B$ is quasiprimitive and that there are $H_1, H_2 \lhd G$, $H=H_1 \times H_2$ with $H \lhd G$ and $[G:H]$ odd. Then $B$ covers the principal blocks $c_i$ of $\cO H_i$, so $B$ also covers the principal block $c\cong c_1 \otimes c_2$ of $\cO H$. Suppose that $C_G(H) \leq H$ and $H=\ker(G\to \Out(c))$. Suppose that $H_1, H_2 \in \{(C_2)^3, J_1, {}^2G_2(q), \SL_2(8)\}$ for some $q=3^{2m+1}$, $m \in \mathbb{N}$. Then the following holds.
\begin{enumerate}[label=(\arabic*)]\setlength\itemsep{0em}
\item If $H_1=\SL_2(8)$ and $H_2=\SL_2(8)$ then $B$ is Morita equivalent to $c$, or to (lxvii), (lxx), (lxxxi).
\item If $H_1=\SL_2(8)$ and $H_2=(C_2)^3$ then $B$ is Morita equivalent to $c$, or to (xxv), (xxx), (xxxii), (xlviii), (lvii), (lxiii), (lxvi), (lxix), (lxxviii).
\item If $H_1=\SL_2(8)$ and $H_2=J_1$ then $B$ is Morita equivalent to $c$, or (lxxx).
\item If $H_1=J_1$ and $H_2=(C_2)^3$ then $B$ is Morita equivalent to $c$, or (lv), (lxiv), (lxxvii).
\item If $H_1=J_1$ and $H_2=J_1$ then $B$ is Morita equivalent to $c$.
\item If $H_1 = {}^2G_2(q)$ and $H_2=\SL_2(8)$, then $B$ is Morita equivalent to $c$ or (lxxxi).
\item If $H_1 = {}^2G_2(q)$ and $H_2=(C_2)^3$, then $B$ is Morita equivalent to $c$ or (lvii), (lxvi), (lxxviii).
\item If $H_1 = {}^2G_2(q)$ and $H_2=J_1$, then $B$ is Morita equivalent to $c$.
\item If $H_1 = {}^2G_2(q)$ and $H_2={}^2G_2(q')$ for a possibly different $q'=3^{2m'+1}$, $m' \in \mathbb{N}$, then $B$ is Morita equivalent to $c$. 
\end{enumerate}
\end{proposition}
\begin{proof}
In cases 1-5 we are dealing with specific groups instead of Morita equivalence classes, so it suffices to examine their outer automorphism groups. Recall $\Out(\SL_2(8))= C_3$, $\Out(J_1)=1$ and $\Out((C_2)^3) = \GL_3(2)$ with maximal subgroup of odd order is $C_7 \rtimes C_3$. 

In the notation of Method \ref{clubsuit6}, we write $T$ for the maximal possible subgroup of odd order in the image of $G/H$ under $\alpha$. 
\begin{enumerate}[label=(\arabic*)]
\item If $H_1=H_2= \SL_2(8)$ then $T=C_3 \times C_3$, so $G \leq \Aut(\SL_2(8)) \times \Aut(\SL_2(8))$. In particular, $B$ is among (xlix), (lxvii), (lxxxii) and (lxx).
\item If $H_1=\SL_2(8)$ and $H_2=(C_2)^3$ then $T=C_3 \times (C_7 \rtimes C_3)$, so in particular $B$ is among (viii), (xxv), (xxx), (xxxii), (xlviii), (lvii), (lxiii), (lxvi), (lxix), and (lxxviii).
\item If $H_1=\SL_2(8)$ and $H_2=J_1$ then $T=C_3$, so $B$ is either in (lxv) or in (lxxx).
\item If $H_1=J_1$ and $H_2=(C_2)^3$ then $T=C_7 \rtimes C_3$, so $B$ is among (xxix), (lv), (lxiv) and (lxxvii).
\item If $H_1=J_1$ and $H_2=J_1$ then $T=1$, so $G=H$ and hence $B$ is in (lxxix).
\end{enumerate}
In cases 6-9 we are dealing with an infinite family of groups, so while for each possibility the group $T$ is known, there could be an arbitrary number of Morita equivalence classes for $B$. From \cite{g2inner} every degree of an irreducible character of the principal block of $H_1$ occurs with multiplicity $1$ or $2$, which implies that every automorphism of the principal block of $H_1$ is inner. Hence since $H=\ker(G \to \Out(c))$ in our situation $\beta(\alpha(\Out_G(H_1))) = 1$ (that is, $G$ acts trivially on $H_1$). So in cases 6-8 we can limit our analysis to the subgroup $\Out(H_2)$, and, in case 9, $G=H$. 

The block $B$ of $G$ is a crossed product of $c$ with $G/H$, and from \cite{ea14} the principal blocks of ${}^2 G_2(q)$ with $q$ as specified above are in the same Morita equivalence class. Hence the possible Morita equivalence classes for $B$ can be determined simply by applying Method \ref{clubsuit6}. We take a chain of normal subgroups $\{N_i\}$ of length $t$, where $N_0:=H$ and $N_t:=G$, and consider the corresponding block chain $\{b_t\}$.
\begin{enumerate}[label=(\arabic*)]
\setcounter{enumi}{5}
\item  If $H_1 = {}^2G_2(q)$ and $H_2=\SL_2(8)$ then $T=C_3$, so the only nontrivial possibility for the Morita equivalence class of $B$ is (lxxxi), realised by the principal block of ${}^2G_2(q) \times \Aut(\SL_2(8))$. There is only one such Morita equivalence class, independent of the choice of $q$ \cite{ea14}.
\item  If $H_1 = {}^2G_2(q)$ and $H_2=(C_2)^3$ then $T=C_7 \rtimes C_3$, so the only nontrivial possibilities for the Morita equivalence class of $B$ are (lvii), (lxvi), (lxxviii), in each case realised by the principal blocks of ${}^2G_2(q) \times ((C_2)^3 \rtimes F)$ for the appropriate choice of $F$.
\item  If $H_1 = {}^2G_2(q)$ and $H_2=J_1$ then $T=1$, so $G=H$. Note that all the possibilities for $c$ are Morita equivalent, since $c = c_1 \otimes \cO J_1$ and all the possibilities for $c_1$ are Morita equivalent, still from \cite{ea14}.
\item  If $H_1 = {}^2G_2(q)$ and $H_2={}^2 G_2(q')$ then $G=H$. Then all the possibilities for $c$ are Morita equivalent, since $c = c_1 \otimes c_2$ and all the possibilities for $c_1$ and $c_2$ are Morita equivalent. \qedhere
\end{enumerate}
\end{proof}

\begin{proposition}\label{321}
Let $G$ be a finite group and $B$ be the principal block of $\cO G$ with defect group $D= (C_2)^6$. Suppose that $B$ is quasiprimitive and that there are $H_1, H_2 \lhd G$, $H=H_1 \times H_2$ with $H \lhd G$ and $[G:H]$ odd. Then $B$ covers the principal blocks $c_i$ of $\cO H_i$, so $B$ also covers the principal block $c\cong c_1 \otimes c_2$ of $\cO H$. Suppose that $C_G(H) \leq H$ and $H=\ker(G\to \Out(c))$. Suppose that $H_1 \in \{J_1, {}^2G_2(q), \SL_2(8)\}$ for some $q=3^{2m+1}$, $m \in \mathbb{N}$, and that $c_2$ is Morita equivalent to $\cO (A_4 \times C_2)$ or to $B_0(\cO (A_5 \times C_2))$. Then the following holds:
\begin{enumerate}[label=(\arabic*)]\setlength\itemsep{0em}
\item If $H_1=\SL_2(8)$ and $c_2$ is Morita equivalent to $\cO(A_4 \times C_2)$ then $B$ is Morita equivalent to $c$, or (lvii), (xxxii).
\item If $H_1=\SL_2(8)$ and $c_2$ is Morita equivalent to $B_0(\cO(A_5 \times C_2))$ then $B$ is Morita equivalent to $c$, or (lviii).
\item If $H_1=J_1$ and $c_2$ is Morita equivalent to $\cO(A_4 \times C_2)$ then $B$ is Morita equivalent to $c$.
\item If $H_1=J_1$ and $c_2$ is Morita equivalent to $B_0(\cO(A_5 \times C_2))$ then $B$ is Morita equivalent to $c$.
\item If $H_1 = {}^2G_2(q)$ and $c_2$ is Morita equivalent to $\cO(A_4 \times C_2)$ then $B$ is Morita equivalent to $c$.
\item If $H_1 = {}^2G_2(q)$ and $c_2$ is Morita equivalent to $B_0(\cO(A_5 \times C_2))$ then $B$ is Morita equivalent to $c$.
\end{enumerate}
\end{proposition}
\begin{proof}
From \cite{eali18}, $\Pic(\cO(A_4 \times C_2)) = S_3 \times C_2$ and $\Pic(B_0(\cO(A_5 \times C_2))) = C_2 \times C_2$. We use Method \ref{clubsuit6}. Moreover, we can apply the same argument as in Proposition \ref{33} to cases 5 and 6 to show that $\beta(\alpha(\Out_G(H_1))) = 1$ and hence we can limit our analysis to the subgroup $\Out(H_2) \leq \Pic(c_2)$. Let $T$ be the maximal subgroup of odd order of $\Pic(c)$ that we need to consider. We take a chain of normal subgroups $\{N_i\}$ of length $t$, where $N_0:=H$ and $N_t:=G$, and consider the corresponding block chain $\{b_t\}$.
\begin{enumerate}[label=(\arabic*)]
\item If $H_1=\SL_2(8)$ and $c_2$ is Morita equivalent to $\cO (A_4 \times C_2)$ then $T=C_3 \times C_3$, so $t \leq 2$. If $G=N$, we are done. If $G/N_0 \cong C_3$, the possibilities for the Morita equivalence class of $B$ are as follows:
	\begin{enumerate}[label=(\alph*)]
	\item (lvii), realised when $N_0=\SL_2(8) \times A_4 \times C_2$.
	\item (viii), realised when $N_0=\SL_2(8) \times \PSL_3(7)$.
	\item A nonprincipal block of $\cO ((\SL_2(8) \times (C_2)^2) \rtimes 3^{1+2}_+) \times C_2)$, realised when $N_0=\SL_2(8) \times \PSL_3(7) \times C_2$.
	\end{enumerate}
From Lemma \ref{iqmagic6} the inertial quotient of $B$ contains $C_{21}$ as a subgroup. Then by Lemma \ref{sambalecalcs6} case (b) cannot occur, since $l(B)=7$, but $k(E) \neq 7$ in all cases where $C_{21} \leq E$. Similarly, case (c) cannot occur, since $l(B)=7$ but $k(E) \neq 7$ whenever $C_{21} \leq E$.

If $G/N \cong C_3\times C_3$, then we consider the group $H' = \Aut(\SL_2(8)) \times H_2$, its principal block $c'$ and $N_0':=H'$. Note that $H' \lhd G$. From $\cite{eali18}$ $\Pic(B_0(\cO(\Aut(\SL_2(8))))) = C_3$, hence the subgroup of $\Pic(c')$ that we need to consider is $T'=C_3 \times C_3$. There are three possible embeddings of $G/N_0' \cong C_3$ in $T'$, which correspond to the following possibilities for the Morita equivalence class of $B$:
	\begin{itemize}
	\item (xxx), realised when $N_0'=\Aut(SL_2(8)) \times \PSL_3(7) \times C_2$.
	\item (xxv), which is realised when $N_0'=\Aut(\SL_2(8)) \times C_3 \times A_4 \times C_2$ (see \cite[4.13]{ar19}).
	\item (xxxii), realised when $N_0'=\Aut(\SL_2(8)) \times C_3 \times \PSL_3(7) \times C_2$ (see \cite[4.13]{ar19}).
	\end{itemize}
From Lemma \ref{sambalecalcs6} none of these possibilities can occur, as the inertial quotient of $B$ must contain $E= C_3 \times (C_7 \rtimes C_3)$ as a normal subgroup and so be equal to $E$, and $k(E)=15$, but in each of these possibilities the number of simple modules is different from $15$.
\item If $H_1=\SL_2(8)$ and $c_2$ is Morita equivalent to $B_0(\cO (A_5 \times C_2))$ then $T=C_3$, so $t \leq 1$. If $G=N$, we are done. Otherwise, by Method \ref{clubsuit6} the only possibility for the Morita equivalence class of $b_1$ is (lviii), realised when $N_0=\SL_2(8) \times A_5 \times C_2$.
\item If $H_1=J_1$ and $c_2$ is Morita equivalent to $\cO (A_4 \times C_2)$ then $T=C_3$, so $t \leq 1$. If $G=N_0$, we are done. Otherwise, by Method \ref{clubsuit6} the only possibility for the Morita equivalence class of $b_1$ would be (xxix). However from Proposition \ref{iqmagic6} the inertial quotient of $b_1$ is $(C_7 \rtimes C_3) \times C_3$ or a group that contains it, and in this case $l(b_1)=5$, but $k(E) \neq 5$ in every possible situation, so this case cannot occur.
\item If $H_1=J_1$ and $c_2$ is Morita equivalent to $B_0(\cO (A_5 \times C_2))$ then $T=1$, so $G=N_0$ and we are done.
\item If $H_1={}^2G_2(q)$ and $c_2$ is Morita equivalent to $\cO (A_4 \times C_2)$ then $T=C_3$, so $t \leq 1$. If $G=N_0$, we are done. Otherwise, by Method \ref{clubsuit6} the only possibility for the Morita equivalence class of $b_1$ would be (xxix). However, with an identical argument as in case (4), this case cannot occur. Then $G=N_0$.
\item If $H_1={}^2G_2(q)$ and $c_2$ is Morita equivalent to $B_0(\cO (A_5 \times C_2))$ then $T=1$, so $G=N_0$ and we are done. \qedhere
\end{enumerate}
\end{proof}

\begin{proposition}\label{222}
Let $G$ be a finite group and $B$ be the principal block of $\cO G$ with defect group $D= (C_2)^6$. Suppose that $B$ is quasiprimitive, and that there are $H_1, H_2, H_3 \lhd G$, $H=H_1 \times H_2 \times H_3$ with $H \lhd G$ and $[G:H]$ odd. Then $B$ covers the principal blocks $c_i$ of $\cO H_i$, so $B$ also covers the principal block $c=c_1 \otimes c_2 \otimes c_3$ of $\cO H$. Suppose that $C_G(H) \leq H$ and $H=\ker(G\to \Out(c))$. Suppose that $H_1, H_2, H_3 \in \{(C_2)^2, A_5, \PSL_2(q)\}$ for some odd $q$. Then each block $c_i$ is source algebra equivalent to either $\cO (C_2)^2$, $\cO A_4$ or $B_0(\cO(A_5))$. For brevity, we call these respectively nilpotent, of type $A_4$ and of type $A_5$. Further, suppose that if $c_i$ is nilpotent then $H_i \cong (C_2)^2$. Then the following holds:
\begin{enumerate}[label=(\arabic*)]\setlength\itemsep{0em}
\item If $c_1$ is nilpotent, $c_2$ is of type $A_4$ and $c_3$ is of type $A_5$, then $B$ is Morita equivalent to $c$ or (xiii), (xvi) or (xxxvi).
\item If $c_1$ is nilpotent and $c_2$ and $c_3$ are of type $A_5$, then $B$ is Morita equivalent to $c$ or (xxxvii).
\item If $c_1, c_2$ are of type $A_4$ and $c_3$ is of type $A_5$, then $B$ is Morita equivalent to $c$.
\item If $c_1$ is of type $A_4$ and $c_2, c_3$ are of type $A_5$, then $B$ is Morita equivalent to $c$.
\item If $c_1, c_2, c_3$ are of type $A_5$, then $B$ is Morita equivalent to $c$.
\end{enumerate}
\end{proposition}
\begin{proof}
We use Method \ref{clubsuit6}, together with the Picard groups computed in \cite{eali18}. The image of $G/H$ under $\gamma\beta\alpha$ is contained in $\Pic(c_1) \times \Pic(c_2) \times \Pic(c_3)$, and we denote its maximal subgroup of odd order by $T$. We take a chain of normal subgroups $\{N_i\}$ of length $t$, where $N_0:=H$ and $N_t:=G$, and consider the corresponding block chain $\{b_t\}$.
\begin{enumerate}[label=(\arabic*)]
\item In this case, $G/H \leq C_3 \times \Pic(c_2) \times \Pic(c_3)$ as $H_1 \cong (C_2)^2$. Then $T=C_3 \times C_3$. From Lemma \ref{iqmagic6} the inertial quotient of $B$ must contain $C_3 \times C_3$ as a normal subgroup, so in particular looking at every possibility listed in Lemma \ref{sambalecalcs6} we have that $l(B) \geq 9$.
If $G=N_0$, we are done. Otherwise, if $G/N_0 \cong C_3$ we have three possibilities for the Morita equivalence class of $b_1$: 
\begin{enumerate}[label=(\alph*)]
\item (xxxv), realised when $N_0=(C_2)^2 \times A_4 \times A_5$.
\item (iii), realised when $N_0=(C_2)^2 \times \PSL_3(7) \times A_5$.
\item A nonprincipal block of $B_0(\cO(((C_2)^4 \rtimes 3^{1+2}) \times A_5))$, realised when $N_0=(C_2)^2 \times \PSL_3(7) \times A_5$.
\end{enumerate}
Note that in cases (b) and (c) $l(b_1)=3$, so these cases cannot occur as principal blocks. 

If $G/N \cong C_3 \times C_3$, then we consider $H_1'=(H_1 \rtimes C_3)$ which now has a block $c_1'$ of type $A_4$. Let $H' = H_1' \times H_2 \times H_3$. Note that $H' \lhd G$. Let $N_0'=H'$: then $G/N_0' \leq C_3$ and $T' \cong C_3 \times C_3  \leq \Pic(c_1') \times \Pic(c_2) \times \Pic(c_3)$, so there are three possible embeddings which determine the following Morita equivalence classes for $B$:
\begin{itemize}
\item (xiii), realised when $N_0'=A_4 \times \PSL_3(7) \times A_5$. 
\item (xvi), realised as a crossed product when $N_0'=\PSL_3(7) \times \PSL_3(7) \times A_5$.
\end{itemize}
Note that the first case occurs with two different embeddings.
\item In this case, $T=C_3$. There is a unique possibility for the Morita equivalence class of $B$, which is (xxxvii), realised when $N_0=(C_2)^2 \times (A_5)^2$.
\item In this case, $T=C_3 \times C_3$ so $t \leq 2$. From Lemma \ref{iqmagic6} the inertial quotient of $b_1$ must contain $(C_3)^3$ as a normal subgroup, so in particular looking at every possibility from Lemma \ref{sambalecalcs6} we have that $l(b_1) =27$ or $l(b_1)= 17$. 

The block $b_1$ is Morita equivalent to a crossed product of the basic algebra of $c$ with $X_1=N_1 /N_0$ as detailed before; let $\omega_1$ be the corresponding homomorphism. Then the possibilities for the Morita equivalence class of $b_1$ are as follows:
\begin{enumerate}[label=(\alph*)]
\item (xiii), realised when $N_0=\PSL_3(7) \times A_4 \times A_5$.
\item (xvi), realised when $N_0=\PSL_3(7)^2 \times A_5$.
\end{enumerate}
In both cases $l(b_1)=9$, a contradiction. Then $G=N_0$ and $B=c$.
\item In this case, $T=C_3$. As in the previous case the inertial quotient of $b_1$ must contain $(C_3)^3$ as a normal subgroup. There is a single nontrivial possibility for the Morita equivalence class of $B$: (xiv), realised when $N_0=\PSL_3(7) \times (A_5)^2$. However, in this case $l(B)=9$, a contradiction.
\item In this case, $T=1$ so $G=N_0$ and $B=c$ and we are done. \qedhere
\end{enumerate}
\end{proof}

\begin{proposition} \label{perm3}
Let $G$ be a finite group and $B$ be a quasiprimitive block of $\cO G$ with defect group $D= (C_2)^6$. Suppose that $H = H_1 \times H_2 \times H_3 \lhd G$, and that $G$ acting by conjugation permutes transitively the set $\{H_i\}_{i=1,2,3}$. Suppose that $[G:H]$ is odd, and let $c$ be the unique block of $\cO H$ covered by $B$. Suppose that $C_G(H) \leq H$, and that $H_1 \cong H_2 \cong H_3 \in \{A_5, \PSL_2(q)\}$ for some odd $q$. Then each block $c_i$ is source algebra equivalent to $\cO A_4$ or $B_0(\cO(A_5))$. For brevity, we call these respectively of type $A_4$ and of type $A_5$. Then
\begin{enumerate}[label=(\arabic*)]
\item If $c_1$ is of type $A_4$ then $B$ is Morita equivalent to $c$ or to $\cO (D \rtimes E)$ for a subgroup of odd order $E \leq \GL_6(2)$.
\item If $c_1$ is of type $A_5$ then $B$ is Morita equivalent to $c$ or to (lxi).
\end{enumerate}
\end{proposition}
\begin{proof}
Let $\sigma: G \to S_3$ be the homomorphism given by the action of $G$ by permutation on the set $\{H_i\}$. Then $N = \ker(\sigma)$ is a normal subgroup of $G$ of index $3$. Let $b_N$ be the unique block of $\cO N$ covered by $B$. 
\begin{enumerate}[label=(\arabic*)]
\item In this case, since each block $c_i$ is source algebra equivalent to $\cO A_4$, then the block $c$ is source algebra equivalent to $\cO (A_4)^3$. In particular, it is basic Morita equivalent to that block, and hence inertial. Then from Proposition \ref{nilpotentcovered6} $B$ is also inertial, and we are done.
\item We compute $\cT(B_0(\cO (A_5)^3))$. From the main theorem of \cite{bkl17}, there is an exact sequence
$$1 \to \Out_D(ici) \to \cT(c) \to \Out(D,\mathcal{F}_c)$$
where $i$ is a source algebra idempotent for $c$, and $\mathcal{F}_c$ is the fusion system of the block. Let $D_j$ be a defect group for $c_j$, the block of $H_j$ covered by $c$. Then $D=D_1 \times D_2 \times D_3$. First, note that the subgroup $\cT(c_1) \wr S_3 \leq \cT(c)$. From Lemma 3.1 in \cite{ea17}, $\Out_D(ici) = \prod_{j=1}^3 \Out_{D_j}(c_j)$ since each block $c_j$ is equal to its source algebra, as seen in Proposition \ref{classification6}. From \cite[1.5]{bkl17} we have that $\Out_{D_j}(c_j)=1$, so $\Out_D(ici)=1$. From Remark 1.2.f in \cite{bkl17}, $\Out(D, \mathcal{F}_c) \cong N_{\GL_6(2)}((C_3)^3)/(C_3)^3 \cong C_2 \wr S_3$. Then the last map in the sequence above is surjective and $\cT(c)=C_2 \wr S_3$. 

Suppose initially that $H=\ker(G\to \Out(c))$. In this case, following Method \ref{clubsuit6}, there is an embedding of $N/H$ in $\Pic(c_1) \times \Pic(c_2) \times \Pic(c_3) = (C_2)^3$. Then $N=H$ since $\Pic(c_i) \cong C_2$ for each $i$ and $G/H$ is odd. Since we have a source algebra equivalence we can limit our analysis to $\cT(c)$ and we do not need to consider the whole $\Pic(c)$. As computed above, $\cT(c)=C_2 \wr C_3$, so there is a unique possibility for the map $\omega: G/N \to \cT(c)$ that specifies this crossed product, which corresponds to $B$ being Morita equivalent to (lxi), realised in the example $(A_5)^3 \lhd (A_5 \wr C_3)$.

Now suppose that $G[c]=\ker(G \to \Out(c))$ strictly contains $H$. From Proposition \ref{inneraut6}, the unique block $\hat{c}$ of $G[c]$ covered by $B$ is source algebra equivalent to $c$. In particular $\cT(\hat{c}) = \cT(c)$. Then we can repeat the argument above to obtain the same result. \qedhere
\end{enumerate}
\end{proof}

\begin{remark} \normalfont Just as in \cite{ar19}, the assumption that $H=\ker(G \to \Out(c))$ in each of the propositions can be dropped, as in any relevant case we can reduce to such a situation without modifying the final deductions on $\gamma(\beta(\alpha(G/H)))$. It is, in fact, sufficient to note that by definition $\ker(G \to \Out(c)) = G[c]_\cO$ (as defined in Proposition \ref{inneraut6}), and that $H \leq G[c]_\cO$. Let $\hat{c}$ be the principal block of $G[c]$, which from Proposition \ref{inneraut6} is source algebra equivalent to $c$. Then we can replace $H$ and $c$ with $G[c]$ and $\hat{c}$ respectively in Propositions \ref{33}, \ref{321} and \ref{222} an obtain the same possible Morita equivalence classes for $B$ (see also Corollary 4.11 in \cite{ar19}).
\end{remark}


\section{Principal blocks with defect group $(C_2)^6$}

We classify the principal blocks with a normal defect group as a separate case.
\begin{theorem} \label{normalD6}
Let $B$ be a block of $\cO G$ for $G$ a finite group with a normal abelian Sylow $2$-subgroup $D \cong (C_2)^6$. Then $B$ is Morita equivalent to $\cO(D \rtimes E)$ where $E$ is a subgroup of $\GL_6(2)$ of odd order acting faithfully on $D$.
\end{theorem}
\begin{proof}From the main theorem of \cite{kul85}, in this situation $B$ is source algebra equivalent to a twisted group algebra $\cO_\alpha (D \rtimes E)$ where $E= N_G(D,e)/C_G(D)$ is its inertial quotient, so in particular $E \leq \GL_6(2)$ and $E$ has odd order. From \cite[6.14]{lin18} the cocycle $\alpha$ corresponds to the one defined in \cite[2.5]{usami97}, so since $B$ is the principal block $\alpha=1$. To get all possibilities for $E$, since the action on $D$ is faithful by definition, it is enough to consider all the conjugacy classes of odd order subgroups of $\Aut(D) = \GL_6(2)$.
\end{proof}

We listed all the possible inertial quotients in Section 3. To show that each inertial quotient corresponds to a distinct class of blocks, it is enough to compute each Cartan matrix, with two exceptions: the pairs $E=(C_7)_2$ and $(C_7)_3$ (see Section $4$), and $(C_7 \rtimes C_3)_3$ and $(C_7 \rtimes C_3)_4$. In these cases computing $\dim(J^2(Z(G)))$ with Magma \cite{magma} shows that the blocks are not Morita equivalent.

\begin{proof}[Proof of Theorem \ref{maintheorem6}]
Let $B$ be the principal block of $\cO G$ for a finite group $G$ with Sylow $2$-subgroup $D=(C_2)^6$, such that $B$ is not Morita equivalent to any of the blocks in the statement of the theorem and such that $([G:O_{2'}(Z(G))],|G|)$ is minimised in the lexicographic ordering. Then $B$ has defect group $D$.

First, we show that these hypotheses on $B$ imply two important facts: \begin{itemize}\setlength\itemsep{0em}
\item[(I)] $B$ is quasiprimitive, that is, for any normal subgroup $N \lhd G$ any block of $\cO N$ covered by $B$ is $G$-stable.
In fact, let $N \lhd G$, and let $b$ be a block of $\cO N$ covered by $B$. We write $I_G(b)$ for the stabiliser of $b$ under conjugation by $G$. Then we can consider the Fong-Reynolds correspondent $B_I$ as in Proposition \ref{fong16}, the unique block of $I_G(b)$ covering $b$ and with Brauer correspondent $B$, that is Morita equivalent to $B$ and shares a defect group with it. In particular, using Brauer's third main theorem, $B_I$ is the principal block of $I_G(b)$ since $B$ is its Brauer correspondent. By minimality, it follows that $I_G(b)=G$, and the same is true for any block of any normal subgroup of $G$.
\item[(II)] If there is a normal subgroup $N \lhd G$ such that $B$ covers a nilpotent block $b$ of $\cO N$, then $N \leq O_2(G)Z(G)$. This follows from minimality and quasiprimitivity, using Corollary \ref{kulshammerpuigcorollary6}.
\end{itemize}

Since the principal blocks of $G$ and $G/O_{2'}(G)$ are isomorphic, by minimality we can assume that $O_{2'}(G)=1$. Then when $B$ covers a nilpotent block $b$ of a normal subgroup $M$, (II) implies that $M \leq O_2(G)$. 

From Proposition $5$\textsc{a-d} in \cite{fongharris}, it holds that
$$N=O^{2'}(G) = S_0 \times \prod_{i=1}^t S_i$$
where $S_0$ is a $2$-group and each $S_i$ is a nonabelian simple group. Note that by definition $[G:N]$ is odd, so $D \leq N$. Also note that $O_2(G)=S_0$, and that $O_p(G)=1$ for each $p \neq 2$. Then the Fitting subgroup $F(G)=S_0$.

Each $S_i$ is, by definition, a component of $G$, and since $G/N$ has odd order and is, therefore, solvable, any component of $G$ is among the $S_i$ listed. Then $E(G)= \prod S_i$, so $F^*(G)=N$ and in particular $C_G(N) \leq N$: therefore $N \lhd G \leq \Aut(N)$. 

Let $b$ be the principal block of $N$, and for $i\geq 1$, let $b_i$ be the principal block of $S_i$. Note that each of these blocks is covered by $B$. Then from \cite[15.1]{alp151} its defect group is of the form $D_i=D \cap S_i$. Since $S_i \cap O_2(G)= \{1\}$, (II) implies that each $b_i$ is not a nilpotent block, so in particular $|D_i|\geq 2^2$. Then since $S_i \cap S_j = \{1\}$, $t \leq 3$.

In the following, whenever examining a block $c$ with defect group $(C_2)^2$, we remark that from Proposition \ref{classification6} there is always a source algebra equivalence between $c$ and one of $\cO (C_2)^2$, $\cO A_4$ or $B_0(\cO A_5)$. For brevity, we will say that $c$ is, respectively, nilpotent or of type $A_4$ or $A_5$. \\

Let $s=\log_2(|S_0|)$. We examine each possibility for $t$ and $s$, using Proposition \ref{walter} to determine the possibilities for each $S_i$. 

Suppose that $t=0$. Then $D=S_0$, so $B$ has a normal defect group. Then Proposition \ref{normalD6} gives a contradiction, since every block already appears in the list.

Suppose that $t=1$. Then $S_1 = E(G)$ and: \begin{itemize}\setlength\itemsep{0em}
\item If $s=0$ then $S_1 = \SL_2(64)$. Since $\Out(\SL_2(64)) = C_6$, then $G/N \leq C_3$. If $G=N$ then $B$ is as in case (li), a contradiction. If $G=N \rtimes C_3$ then $B$ is as in case (lxxv), again a contradiction.
\item If $s=1$ then $S_1 = \SL_2(32)$. Since $\Out(\SL_2(32)) = C_5$, then $G/N \leq C_5$. If $G=N$ then $B$ is as in case (xlii), a contradiction. If $G=N \rtimes C_5$ then $B$ is as in case (lxxiii), again a contradiction.
\item If $s=2$ then $S_1 = \SL_2(16)$. Since $\Out(\SL_2(16)) = C_4$, then $G=N$ and $B$ is as in case (xxi), a contradiction.
\item If $s=3$ then $S_1 = \SL_2(8), J_1$ or ${}^2G_2(3^{2m+1})$. Then Proposition \ref{33} gives a contradiction, as $B$ already appears in the list.
\item If $s=4$ then $S_1 = A_5$ or $S_1 = \PSL_2(q)$ for some odd $q$. If $b_1$ is of type $A_4$ then $b$ is inertial, so Proposition \ref{nilpotentcovered6} implies that $B$ is also inertial, and then Proposition \ref{normalD6} gives a contradiction. If $b_1$ is of type $A_5$ then we can apply Method \ref{clubsuit6}, noting that the fact that $\Pic(b_1)=C_2$ (see \cite{bkl17}) implies that the image of $G/N$ under the map $\gamma\beta\alpha$ is contained in $\Out(S_0)$, and therefore $B= b_1 \otimes B'$ where  $B'$ is a block with normal defect group $(C_2)^4$. Now applying the main theorem of \cite{ea17} leads to a contradiction, as all such Morita equivalence classes already appear in our list.
\item If $s=5$ then $b_1$ is nilpotent, so (II) implies that $S_1 \leq S_0$, a contradiction.
\item If $s=6$ then $t=0$, a contradiction.
\end{itemize}

Suppose that $t=2$. Then $S_1 \times S_2= E(G)$ and $s \leq 2$ because each $D_i$ has at least order four. Without loss of generality we suppose that $|D_1| \geq |D_2|$. Then: \begin{itemize}\setlength\itemsep{0em}
\item If $s=0$ and $|D_1| = 2^4$ then $S_1 = \SL_2(16)$ and $S_2 = A_5$ or $S_2 = \PSL_2(q)$ for some odd $q$. Since $\Out(\SL_2(16)) = C_4$, a $2$-group, then $G/N \leq \Out(S_2)$, so $G \leq S_1 \times \Aut(S_2)$. In particular, $B = b_1 \otimes b_2$ where $b_1$ is the principal block of $\SL_2(16)$, and $b_2$ has defect group $(C_2)^2$, and hence $B$ is Morita equivalent to (xxi), (xlv) or (xlvi), a contradiction.
\item If $s=0$ and $|D_1| = 2^3$ then $S_1, S_2 \in \{\SL_2(8), J_1, {}^2G_2(3^{2m+1})\}$. Then Proposition \ref{33} gives a contradiction, as $B$ already appears in the list.
\item If $s=1$ then $S_1 \in \{\SL_2(8), J_1$, ${}^2G_2(3^{2m+1})\}$, and  $S_2 = A_5$ or $S_2 = \PSL_2(q)$ for some odd $q$. Then Proposition \ref{33} applied considering $H_2 = S_0 \times S_2$ gives a contradiction, as $B$ already appears in the list.
\item If $s=2$ then $S_i =A_5$ or $S_i =\PSL_2(q_i)$ for odd $q_i$, $i=1,2$. If each $b_i$ is of type $A_4$ then $b$ is inertial, so Proposition \ref{nilpotentcovered6} implies that $B$ is also inertial, a contradiction because of Proposition \ref{normalD6}. Otherwise Proposition \ref{222} gives a contradiction, as the Morita equivalence class of $B$ already appears in the list.
\end{itemize}
Finally suppose that $t=3$. Then $s=0$, and $E(G)=S_1 \times S_2 \times S_3$, where each $D_i = (C_2)^2$ and each $b_i$ is not nilpotent. The group $G$ acts on $E(G)$ by possibly permuting the three components, giving a homomorphism $\sigma: G \to \operatorname{Sym}(3)$. Let $M=\ker(\sigma)$. Since $N \leq M$, then either $G=M$ or $[G:M]=3$.

If $G=M$ then each $S_i$ is a normal subgroup of $G$, and we can apply Proposition \ref{222} to obtain a contradiction, as the Morita equivalence class of $B$ already appears in the list. If $[G:M]=3$ then we are in the situation described in Proposition \ref{perm3}, so $B$ is Morita equivalent to a block with a normal defect group or to (lxi), again a contradiction.

Therefore, in every possible case we have a contradiction to $B$ being a minimal counterexample, and we are done.

To see that the classes are distinct it is enough to compute the Cartan matrices for each block, a process that produces distinct matrices in each case except for the pairs (ix), (x) and (xxxi), (xxxiii), which can be distinguished by computing $\operatorname{dim}(J^2(Z(B)))$ as observed in Proposition \ref{normalD6}. 

To see that in our case the inertial quotients are invariant under Morita equivalences, note that from Theorem 4.33 in \cite{benroq}, the principal block of $\cO G$ is derived equivalent to its Brauer correspondent in $\cO N_G(D)$ (and, trivially, they have the same inertial quotient). The latter is a principal block with a normal defect group so, using Proposition \ref{normalD6}, it is source algebra equivalent to $\cO (D \rtimes E)$ where $E$ is the inertial quotient of $B$. So if $C$ is the principal block of $\cO H$ then there is a derived equivalence between $\cO(D \rtimes E)$ and $\cO(D \rtimes E')$ where $E'$ is the inertial quotient of $C$. Using Magma \cite{magma}, we verified computationally that each block $c=\cO(D \rtimes E)$ is uniquely determined by the triple $(k(c), l(c),Z(c))$ (recall that derived equivalences induce an isomorphism between the centers), which implies that $E \cong E'$ via a map that preserves the action on $D$.

The last claim follows from Corollary 1.6 in \cite{lin18p}, as the isomorphism class of an elementary abelian defect group is always invariant under Morita equivalences. 
 \qedhere 
\end{proof}


\section{Purely nonprincipal Morita equivalence classes of blocks} \label{pnpsection}
Since Method \ref{clubsuit6} involves crossed products, it considers the blocks involved just as algebras, and as a result of this we lose the ability to distinguish some group-theoretic properties of blocks: among these, distinguishing principal blocks from nonprincipal blocks. In Section 4 we saw some examples of blocks that can only arise as nonprincipal blocks. The aim of this section is to give a survey of the currently known examples which, together with the blocks determined in Theorem \ref{maintheorem6}, we conjecture to be a complete list of representatives of blocks with defect group $(C_2)^6$.

Given a Morita equivalence class of blocks of finite groups, we say that it is \textit{purely nonprincipal} if it contains no principal blocks of finite groups. Examples of these classes are case (a) in the main theorem of \cite{ea17}, and the classes (a), (b) and (c) in the main theorem of \cite{ar19}. Clearly, all the classes of blocks with defect group $(C_2)^6$ that do not appear in Theorem \ref{maintheorem6} are purely nonprincipal.

As we mentioned in the proof of Proposition \ref{normalD6}, any block with a normal defect group $D$ is source algebra equivalent to a twisted group algebra $\cO_\alpha (D \rtimes E)$ where $E$ is the inertial quotient, and when $B$ is principal the twist $\alpha = 1$. It follows that every nontrivial twisted group algebra not Morita equivalent to an untwisted one is in a purely nonprincipal Morita equivalence class. We also saw that $\alpha$ can be chosen as $\beta^{-1}$ where $\beta \in H^2(E,k^\times)$. In particular, to find some examples of nonprincipal blocks with a normal defect group we can look at cases when the inertial quotient has a nontrivial Schur multiplier, and take central extensions. In this normal defect group case whenever $l(B) \neq k(E)$ the block has to be in a purely nonprincipal Morita equivalence class, to not contradict the main theorem of \cite{fongharris}. We constructed the following examples:

\begin{example} \label{nonprincipala} \normalfont
Let $G$ be a block with defect group $(C_2)^4$ Morita equivalent to case (a) in the main theorem of \cite{ea17}. Then: 
\begin{itemize}
\item The block $B_1$ of $G \times (C_2)^2$ has one simple module, but since $k(B_1) = 32$ and $k( (i) )=64$ then it is in a new Morita equivalence class.
\item The block $B_2$ of $G \times A_4$ has the same Cartan matrix as (ii), but $k(B_2)=32$, so it is in a new Morita equivalence class.
\item The block $B_3$ of $G \times A_5$ has the same Cartan matrix as (iii)  but $k(B_3)=32$, so it is in a new Morita equivalence class.
\end{itemize}
\end{example}

\begin{example} \label{nonprincipalb} \normalfont
We can use cases (b) and (c) in the main theorem of \cite{ar19} to construct two purely nonprincipal Morita equivalence classes, by considering the direct product of any representative of each class with $C_2$. These blocks have Cartan matrices that do not appear as the Cartan matrix of any of the blocks in Theorem \ref{maintheorem6}, so they form distinct Morita equivalence classes and are, hence, purely nonprincipal.
\end{example}

\begin{example}\label{nonprincipalc} \normalfont
Let $q \neq 2$ be a prime. A general phenomenon observed in the previous constructions of purely nonprincipal Morita equivalence classes is that some of them appear in central extensions of groups of type $M= (D_1 \rtimes E_1 \rtimes C_q) \times (D_2 \rtimes E_2 \rtimes C_q)$ by a third $C_q$, so for the moment we focus on blocks with a normal defect group. By replicating this process in all possible cases, we found some purely nonprincipal blocks, as follows: \begin{enumerate}
\item When $M=((C_2)^4 \rtimes C_{15}) \times A_4$, a central extension by $C_3$ gives the group $(C_2)^6 \rtimes (C_5 \times 3^{1+2}_+)$, which has two nonprincipal blocks with $5$ simple modules and the same Cartan matrix as (iv) in Theorem \ref{maintheorem6}, but $24$ irreducible characters (while (iv) has $32$).
\item When $M=((C_2)^3 \rtimes C_7)^2$, a central extension by $C_7$ gives the group $(C_2)^6 \rtimes 7^{1+2}_+$, which has six nonprincipal blocks with $1$ simple module and $16$ irreducible characters, so it is distinct from (i) in Theorem \ref{maintheorem6} and Example \ref{nonprincipala} both.
\item When $M=((C_2)^3 \rtimes (C_7 \rtimes C_3))^2$, a central extension by $C_3$ gives the group $(C_2)^6 \rtimes ((C_7 \times C_7) \rtimes 3^{1+2}_+)$, which has two nonprincipal blocks with $7$ simple modules and a new Cartan matrix (meaning it is not the Cartan matrix of any of the blocks of Theorem \ref{maintheorem6}).
\end{enumerate}
\end{example}
In all of these cases, we claim that for each fixed group all the nonprincipal blocks are Morita equivalent, and that considering the central extension $3^{1+2}_-$ instead of $3^{1+2}_+$ gives nonprincipal blocks in the same Morita equivalence class. To prove this, note that $q^{1+2}_+ = (C_q \times C_q) \rtimes C_q$, and $q^{1+2}_- = C_{q^2} \rtimes C_q$, so each central extension $G$ contains a normal subgroup $N = (D_1 \rtimes E_1 \rtimes C_q) \times (D_2 \rtimes E_2) \times C_q$ or $N= (D_1 \rtimes E_1 \rtimes C_{q^2}) \times (D_2 \rtimes E_2)$, so in each case, each nonprincipal block of $\cO G$ covers a block $b$ of $\cO N$ that is Morita equivalent to a block of $(D_1 \rtimes E_1 \rtimes C_q) \times (D_2 \rtimes E_2)$. Thus, following Method \ref{clubsuit6}, $B$ is Morita equivalent to a crossed product of the basic algebra of $b$ and $C_q$, specified by an element of the image of $G/N$ in $\Pic(b)$. Since $N$ is a direct product, the image of $G/N$ is actually contained in $\Pic(b_1) \times \Pic(b_2)$, where $b= b_1 \otimes b_2$. If there is only one subgroup $C_q \times C_q$ contained in $\Pic(b_1) \times \Pic(b_2)$, then there are three possibilities for $B$: two of them are given by $C_q = G/N \leq \Pic(b_i)$ where $i=1,2$, and the third one by the diagonal embedding. Since each extension we considered does not fix the irreducible characters neither of $N_1$ nor of $N_2$, in each case the element in $\Pic(b)$ corresponding to the Morita equivalence class of $B$ is the one given by the diagonal embedding of $C_q$. In particular, there is a unique possibility for the Morita equivalence class of $B$, as we claimed. It remains to show that for each case the Sylow $q$-subgroup of $\Pic(b_i)$ is $C_q$: for our three examples this can be deduced from the Picard groups computed in \cite{eali18} and the main theorem of \cite{liv19}.

\begin{example} \normalfont
When $M=(C_2)^4 \rtimes (C_3)_2 \times A_4$ (with the notation of \cite{ea17}), a central extension by $C_3$ gives the group $(C_2)^6 \rtimes 3^{1+2}_+$ (distinct from the one appearing in (xxxix), where $3^{1+2}_+$ acts faithfully on $(C_2)^6$), which has two nonprincipal blocks with $1$ simple module but $24$ irreducible characters (while (iv) has $32$), which means that the blocks form a distinct Morita equivalence class from any other class with one simple module seen so far. We conjecture that also in this case choosing the other central extension ($3^{1+2}_-$) gives the same Morita equivalence class of blocks.
\end{example}

\begin{example} \normalfont
The only case in which the Schur multiplier is not a cyclic group that we encounter is when $M=(A_4)^3$, since in this case $M(E)=(C_3)^3$. Since $M$ is not a perfect group, there is not a universal central extension, but all central extensions by $M(E)$ lie in the same isoclinism equivalence class. One example of such an extension is given by $G=(C_2)^6 \rtimes \operatorname{SmallGroup}(729,122)$, where $\operatorname{SmallGroup}(729,122)=((C_3)^2 \times 3^{1+2}_+) \rtimes C_3$ and $(C_3)^2$ acts trivially on $(C_2)^6$. This group has twelve nonprincipal blocks with $24$ irreducible characters, six with $32$ irreducible characters and eight with $16$ irreducible characters: any two blocks with the same number of irreducible characters share Cartan matrices and every other invariant. Each of these blocks is purely nonprincipal, since the triple $(k(c),l(c),\operatorname{CartanMatrix}(c))$ does not occur in any of the blocks in Theorem \ref{maintheorem6}. We conjecture that these blocks form three Morita equivalence classes, and explicit computation of examples suggests that, moreover, any block of a central extension of $(A_4)^3$ by $(C_3)^k$ for $k=1,2,3$ is Morita equivalent to a block of $\cO G$. In particular, we conjecture that each nonprincipal block with $32$ irreducible characters of $\cO G$ is Morita equivalent to the block $B_2$ in Example \ref{nonprincipala}. 
\end{example}

\begin{example} \normalfont
In any of the constructions above, whenever we have considered the group $((C_2)^n \rtimes C_{2^n-1}) \rtimes C_k$  where $k$ divides $n$, we could have instead considered $\SL_2(2^n) \rtimes C_k$. In the above, this arose only when $n=3$, but for bigger defect groups the same construction could be replicated (for instance, $(\SL_2(32) \rtimes C_{31})^2 \rtimes 5^{1+2}_+$ gives an example of a nonprincipal block with defect group $(C_2)^{10}$). 
By substituting $(C_2)^3 \rtimes C_7$ with $\SL_2(8)$, we obtain an additional example of a purely nonprincipal block in the group $\SL_2(8)^2 \rtimes 3^{1+2}_+$. Note that the same process constructs one of the blocks in Example \ref{nonprincipalb}.
\end{example}

We propose the following conjecture:
\begin{conjecture}
The purely nonprincipal Morita equivalence classes of blocks with defect group $(C_2)^6$ are the following. For each group listed, there is only one relevant Morita equivalence class of nonprincipal blocks:
\begin{enumerate}[label=(\alph*)]\setlength\itemsep{0em}
\item $((C_2)^4 \rtimes 3^{1+2}_+) \times (C_2)^2$
\item $((C_2)^4 \rtimes 3^{1+2}_+) \times A_4$
\item $((C_2)^4 \rtimes 3^{1+2}_+) \times A_5$
\item $((C_2)^5 \rtimes (C_7 \rtimes 3^{1+2}_+)) \times C_2$
\item $((\SL_2(8) \times (C_2)^2) \rtimes 3^{1+2}_+) \times C_2$
\item $(C_2)^6 \rtimes (C_7 \rtimes 3^{1+2}_+)_2$ 
\item $(C_2)^6 \rtimes (C_5 \times 3^{1+2}_+)$
\item $(C_2)^6 \rtimes 7^{1+2}_+$ 
\item $(C_2)^6 \rtimes 3^{1+2}_+$
\item $(C_2)^6 \rtimes ((C_3)^2 \rtimes 3^{1+2}_+) = (C_2)^6 \rtimes \operatorname{SmallGroup}(243,3)$
\item $(C_2)^6 \rtimes \operatorname{SmallGroup}(729,122) b_1$ 
\item $(C_2)^6 \rtimes \operatorname{SmallGroup}(729,122) b_2$ 
\item $(C_2)^6 \rtimes \operatorname{SmallGroup}(1029,12)$ 
\item $(C_2)^6 \rtimes ((C_7 \times C_7) \rtimes 3^{1+2}_+)$        
\item $\SL_2(8)^2 \rtimes 3^{1+2}_+$
\end{enumerate}
In particular, any block of a finite group with defect group $(C_2)^6$ is Morita equivalent to a block listed above or to a block in the list of Theorem \ref{maintheorem6}. 
\end{conjecture}
The blocks that appear in the list above but were not mentioned in any of the examples were simply obtained by taking a central extension of groups $D \rtimes E$ with $M(E)=C_3$ or $C_7$, depending on the cases. In case (f), the subscript means that the inertial quotient acts as in case (lix) of Theorem \ref{maintheorem6}. 

A proof of this conjecture is currently out of reach, as the method used in \cite{ar19} for $(C_2)^5$ requires data that we do not currently have. In particular, when dealing with arbitrary blocks we have to use the main theorem of \cite{ekks} instead of Walter's result in Proposition \ref{walter}, which requires consideration of more classes of blocks. Moreover, the generalised Fitting subgroup of a group $G$ with a block that is a minimal counterexample has a more complicated structure that involves central products instead of direct products. Finally, we cannot assume that the defect group of our block is the Sylow $2$-subgroup of $G$, so we have to study normal subgroups of index $2$, which requires the theory of $(G,B)$-local systems, originally developed by Usami and Puig in \cite{puigusami93} (see Section 3 in \cite{ar19}). 

All of this is necessary because there is no current ``good" characterization of purely nonprincipal blocks, so we have to look at the general case of an arbitrary block in order to classify all Morita equivalence classes. 

Nevertheless, in all known examples the only purely nonprincipal Morita equivalence classes are the ones with a Brauer correspondent source algebra equivalent to a twisted group algebra with a nontrivial twist. Brou\'e's abelian defect group conjecture implies that all such blocks cannot be principal, but we ask the inverse question: is there any purely nonprincipal Morita equivalence class of blocks that has a Brauer correspondent source algebra equivalent to a group algebra? A proof of Brou\'e's conjecture would give a negative answer for derived equivalence classes, but perhaps some purely nonprincipal Morita example could still be found in the derived equivalence class of a principal block.

\section*{Acknowledgements}
I thank Charles Eaton, my PhD supervisor, for his constant support and for many helpful discussion. I also thank Elliot McKernon for many helpful discussions.


\end{document}